\documentclass{article}
\usepackage{amsmath,amsthm,amssymb,fullpage}

\usepackage[pdftex,
pdfpagemode=FullScreen,
bookmarksopen = true,
breaklinks = true,
hyperfigures,
urlcolor = magenta,
urlcolor = blue,
anchorcolor = green,
hyperindex = true,
colorlinks = true,
linkcolor = blue,
citecolor = red
]{hyperref}

\newtheorem{theorem}{Theorem} 
\newtheorem{observation}{Observation}
\newtheorem{proposition}{Proposition} 

\newtheorem{lemma}{Lemma} 
\newtheorem{definition}{Definition}

\usepackage{titling}
\thanksmarkseries{arabic}

\begin{document}

\title{
Branchwidth is $(1,g)$-self-dual\thanks{
Emails of authors:\!\texttt{
{george.kontogeorgiou@warwick.ac.uk},
{livaditisalex@gmail.com},
{kostaspsa@math.uchicago.edu},
{giannos.stamoulis@lirmm.fr},
{dzoros@math.uoa.gr}}
}
}

\date{}

\author{
Georgios Kontogeorgiou\thanks{Mathematics Institute, University of Warwick, CV4 7AL, UK}
\and
Alexandros Leivaditis\thanks{Mathematical Institute, Universiteit Leiden, Netherlands}
\and
Kostas I. Psaromiligkos\thanks{Department of Mathematics, University of Chicago, Chicago, IL, USA}
\and
Giannos Stamoulis\thanks{LIRMM, Univ Montpellier, CNRS, Montpellier, France and Inter-university Postgraduate Programme “Algorithms, Logic, and Discrete Mathematics” (ALMA). Supported by the French-German Collaboration ANR/DFG Project UTMA (ANR-20-CE92-0027).}
\and Dimitris Zoros\thanks{Department of Mathematics, National and Kapodistrian University of Athens, Athens, Greece.}
 }

\maketitle

\begin{abstract}
A graph parameter is self-dual in some class of graphs embeddable in some surface if its value does not change in the dual graph by more than a constant factor.
We prove that the branchwidth of connected hypergraphs without bridges and loops that are embeddable in some surface of Euler genus at most $g$  is an $(1,g)$-self-dual parameter. This is the first proof that branchwidth is an additively self-dual width parameter.
\end{abstract}


\section{Introduction}
A \emph{surface} is a connected compact 2-manifold without boundaries. A surface $\Sigma$ can be obtained, up to homeomorphism, by adding $\mathbf{eg}(\Sigma)$ \emph{crosscaps} to the sphere and $\mathbf{eg}(\Sigma)$
is called the \emph{Euler genus} of $\Sigma$.
An \emph{embedding of a graph} $G$ on a surface $\Sigma$ is a drawing of $G$ on $\Sigma$ without edge crossings. The \emph{Euler genus} of a graph is the smallest number $g\in\mathbb{N}$ such that $G$ can be embedded on a surface of Euler genus $g$.
See the monograph by Mohar and Thomassen~\cite{MoharT01grap} for more on graphs on surfaces.

The \emph{dual graph} of a graph $G$ embedded on a surface $\Sigma$ is the graph $G^*$ embedded on $\Sigma$ that is obtained by considering a vertex $f^*$ for each face $f$ of $G$ and
an edge $e^*$ for every edge $e$, where $e^* = \{f_1^*,f_2^*\}$ and $f_1, f_2$ are the faces of $G$ to which $e$ is incident (if $f_1 = f_2$, then $e^*$ is a loop). Note that
$G^*$ can be embedded on $\Sigma$ by drawing each vertex $f^*\in V(G^*)$ in the interior of the face $f$ of $G$
and every $e^*\in E(G^*)$ crossing the edge $e$ of $G$ exactly once. We call $e^*$ the \emph{dual edge} of $e$.

\begin{definition}		
Let $\Sigma$ be a surface and $\Lambda_{\Sigma}$ the set of all  finite graphs embedded on $\Sigma$. A function $f:\Lambda_{\Sigma}\rightarrow{\mathbb N}$  is called $(a,b)$-self-dual if for every $G\in \Lambda_{\Sigma}$ it holds that $f(G^*)\leq a\cdot f(G)+b$. 
\end{definition}

Self-duality of graph parameters like treewidth, pathwidth, or branchwidth has played a fundamental role in the proof of Wagner's conjecture by Robertson and Seymour~\cite{RobertsonS04GMXX} and has been used to design polynomial-time approximation algorithms for these parameters~\cite{BodlaenderF02appr}.

Treewidth has attracted most of the research concerning self-duality of graph parameters. Lapoire proved~\cite{Lapoire96tree}, using algebraic methods, that treewidth is $(1,1)$-self-dual in planar graphs, settling a conjecture stated by Robertson and Seymour~\cite{RobertsonS84GMIII}. Shorter proofs of this result were given by Bouchitté, Mazoit, and Todinca~\cite{BouchitteMT03chor} and Mazoit~\cite{Mazoit04deco}, exploiting the properties of minimal separators in planar graphs.
Later, Mazoit~\cite{Mazoit12tree} showed that treewidth is $(1,g + 1)$-self-dual in graphs of Euler genus at most $g$.

In the case of pathwidth, Fomin and Thilikos~\cite{FominT07onse} proved that it is a $(6, 6g - 2)$-self-dual parameter in graphs polyhedrically embedded on surfaces of Euler genus at most $g$. This result was improved for planar graphs by Amini, Huc, and Pérennes~\cite{AminiHP09onth} who showed that pathwidth is $(3, 2)$-self-dual in 3-connected planar graphs and $(2, 1)$-self-dual in planar graphs with a Hamiltonian path.

Concerning branchwidth (see~\autoref{subsec:bw} for a formal definition of the branchwidth of a graph), Seymour and Thomas~\cite{SeymourT94call} proved that it is $(1, 0)$-self-dual in planar graphs that are not forests (for more direct proofs, see also~\cite{MazoitT07bran} and~\cite{HicksM07theb}).
For graphs of Euler genus at most $g$,
Sau and Thilikos~\cite{SauT11onse}
proved that branchwidth is $(6, 2g - 4)$-self-dual in graphs of Euler genus at most $g$.
As mentioned in~\cite{SauT11onse}, one can show the $(\frac{3}{2}, g+2)$-self-duality of branchwidth in graphs of Euler genus at most $g$, by combining using the $(1,g + 1)$-self-duality of treewidth in graphs of Euler genus at most $g$~\cite{Mazoit12tree} and the fact that for every graph $G$ with at least three edges, $\mathbf{bw}(G) \leq \mathbf{tw}(G) + 1 \leq \frac{3}{2}\mathbf{bw}(G)$~\cite{RobertsonS91GMX}, where $\mathbf{bw}(G)$ and $\mathbf{tw}(G)$ denote the branchwidth and the treewidth of $G$, respectively.
Sau and Thilikos conjectured~\cite{SauT11onse} that branchwidth is $(1, g)$-self-dual.

We prove this conjecture for connected graphs of Euler genus at most $g$ without bridges and loops. A \emph{bridge} of a graph is an edge whose deletion disconnects the graph, while a \emph{loop} is an edge whose endpoints are the same vertex.

\begin{theorem}\label{thm:main}
The branchwidth of connected graphs of Euler genus at most $g$ without bridges and loops is $(1,g)$-self-dual.\end{theorem}

The proof of~\autoref{thm:main} is presented in~\autoref{sec:selfdual-graphs}.
Our approach is to first show that the \emph{matroid branchwidth} of a connected bridgeless graph of Euler genus at most $g$ is $(1,g)$-self-dual. 
The matroid branchwidth of a graph $G$ is a parameter defined in~\cite{MazoitT07bran}, which we denote by $\mu\text{-}\mathbf{bw}(G)$, and evaluates over branch decompositions of $G$, whose width is not measured in terms of border size of a bipartition of the edge set (like branchwidth) but using the \emph{matroid connectivity function}~\cite{Oxley11matr} on this bipartition (see~\autoref{subsec:mbw} for formal definitions). 
After proving the  $(1,g)$-self-duality of matroid branchwidth on connected bridgeless graphs of Euler genus at most $g$, to prove~\autoref{thm:main} one can use the fact that
$\mathbf{bw}(G) = \mu\text{-}\mathbf{bw}(G)$, if $G$ is connected and bridgeless~\cite{MazoitT07bran} (see~\autoref{prop:bw-mbw}). \autoref{thm:main} is stated for loopless graphs because, in order to guarantee that $\mathbf{bw}(G^*) = \mu\text{-}\mathbf{bw}(G^*)$, one has to ensure that $G^*$ is bridgeless, and the latter is true if $G$ has no loops (see~\autoref{lem:no-bridge-dual}).

\paragraph{From graphs to hypergraphs.}
\autoref{thm:main} also extends to hypergraphs. Before stating the result, let us give some additional definitions on embeddings of hypergraphs and their duals.
A \emph{hypergraph} is a pair $(V,E)$ of sets,
where the elements of $E$ are non-empty subsets
of $V$.
We call the set $E$ the set of \emph{hyperedges} of $(V,E)$.

An \emph{open arc} $A$ on a surface $\Sigma$ is a subset of $\Sigma$ homeomorphic to $(0,1)$. The set of its \emph{endpoints} is defined to be the set $\bar{A}\setminus A$, where $\bar{A}$ denotes the closure of $A$. 
A connected subset $X$ of a surface $\Sigma$ is a \emph{star} if it contains a point $v_x$ called its \emph{center} such that $X\setminus \{v_x\}$ is a union of pairwise disjoint open arcs called \emph{half-edges}. The set of \emph{endpoints of a star} is the set $\bar{X}\setminus X$.  
An \emph{embedding $\Lambda$ of a hypergraph $G$} on a surface $\Sigma$ is a subset of $\Sigma$ consisting of two disjoint sets, $V_G$ and $E_G$, where $V_G$ is a set of points of $\Sigma$ associated bijectively with the vertices of $G$ and $E_G$ is a set of pairwise disjoint stars associated bijectively with the hyperedges of $G$ such that
the endpoints of a star are associated with the elements of the hyperedge that is associated with said star.
We will be using a common symbol for both the hypergraph and its embedding. A \emph{face} of an embedded hypergraph $G$ is a connected component of $\Sigma\setminus G$.

The \emph{dual of a hypergraph} $G$ embedded on $\Sigma$ is a hypergraph $G^*$ which has exactly one vertex lying in the interior of each face of $G$ and exactly one star $s^*$ aligned over each star $s$ of $G$ such that their centres coincide and the endpoints of $s^*$ (which may in pairs coincide) are the vertices of $G^*$ lying in those faces of $G$ that are incident to $s$ (these also may in pairs coincide). We now state our main result in its full generality.

\begin{theorem}
\label{main_result}
The branchwidth of connected hypergraphs, without loops and bridges, that are embeddable on a surface of Euler genus at most $g$ is $(1,g)$-self-dual.
\end{theorem}

The proof of~\autoref{main_result} uses~\autoref{thm:main} as a base case. In fact, we first show that the branchwidth of a hypergraph and its \emph{incidence graph} are the same (\autoref{lemma1}) and the same holds after considering its dual (\autoref{lemma2}). Then, \autoref{main_result} follows by applying~\autoref{thm:main} for the incidence graph.

\section{\texorpdfstring{$(1,2g)$-self-duality of branchwidth on graphs}{(1,2g)-self-duality of branchwidth on graphs}}
\label{sec:selfdual-graphs}

In this section we aim to prove~\autoref{thm:main}.
We first give the formal definitions of a branch decomposition of a symmetric set function, of branchwidth, and of the matroid branchwidth of a graph (\autoref{subsec:bw} and \autoref{subsec:mbw}).
We then show how the matroid branchwidth of a graph changes when removing or contracting certain edges under different connectivity assumptions (\autoref{subsec:mbw-and-operations}).
Using these results, in~\autoref{subsec:dual-bridge-border}, we show that the contraction of a non-loop edge of a given graph (which corresponds to the removal of its dual edge in the dual graph) preserves the same ``gap'' between the matroid branchwidth of a graph and the one of its dual, under the assumption that the dual edge of the contracted edge is not a bridge of the dual graph.
We conclude this section with the proof of~\autoref{thm:main}, that is presented in~\autoref{subsec:proof-graphs}.
At this point, we would like to stress that, in this paper, we allow graphs
to contain multiple edges and/or loops.

\subsection{Branch decompositions and branchwidth}
\label{subsec:bw}

We first define branch decompositions of symmetric set functions. Here, we use terminology from~\cite{GeelenGW02bran} and~\cite{Vatshelle12neww}.
\paragraph{Branch decompositions.}
Given a set $A$, we use $2^A$ to denote the set of all its subsets.
Let $A$ be a finite set.
A \emph{branch decomposition} on $A$ is a pair $(T,\lambda)$, where
$T$ is a binary tree and $\lambda$ is a bijection from $A$ to the set of its leaves, denoted by $L(T)$.
For every edge $e\in E(T)$, let $T_1,T_2$ be the two connected components of $T\setminus e$.
Notice that if we set $A_i^e = \{\lambda(v)^{-1} \mid v\in L(T)\cap V(T_i)\}$ for $i\in\{1,2\}$,
then $(A_1^e, A_2^e)$ is a partition of $A$.
Let $f : 2^A \to \mathbb{R}$ be a symmetric set function (i.e., for every $B\in 2^{A}$, $f(B) = f(A\setminus B)$).
For every $e\in E(T)$, we define $f_\lambda(e):=f(A_1^e)$ (note that this is well-defined since $f$ is symmetric).
The \emph{width of $f$} on a branch decomposition $(T,\lambda)$ on $A$, denoted by $w_f (T,\lambda)$, is the maximum over all $e \in E(T)$ of $f_\lambda(e)$.
The \emph{branchwidth of $f$ on $A$} is the minimum width of $f$ over all branch decompositions $(T,\lambda)$ on $A$.

\paragraph{Branchwidth.}
Given a hypergraph $G$ and a set $E\subseteq E(G)$,
we define the \emph{border} of $E$ as the largest subset of $V(G)$ whose each element is the end of both an edge of $E$ and an edge of $E(G)\setminus E$ and is denoted by $\partial(E)$.
We use $\delta$ to denote the function that maps each $F\subseteq E(G)$ to $|\partial(F)|$.
The \emph{branchwidth} of $G$ is the branchwidth of $\delta$ on $E(G)$ and is denoted by $\mathbf{bw}(G)$.
We will use the term \emph{branch decomposition of $G$} to refer to a branch decomposition of $E(G)$.

\paragraph{Connected branch decompositions.}
A branch decomposition is \emph{connected} when for every $e\in E(T)$, both $A_1^e$ and $A_2^e$ induce a connected subgraph of $G$. We will use the following result~\cite[Theorem 1]{MazoitT07bran} (see also~\cite[Lemma 1]{BarriereFFFNST12conn} for a proof for graphs).
\begin{proposition}\label{prop:conn_bw}
Every connected bridgeless hypergraph $H$ has a connected branch decomposition of width $\mathbf{bw}(H)$.
\end{proposition}

\paragraph{Duals of branch decompositions.}
Let $G$ be a graph embedded on a surface.
The \emph{dual of a branch decomposition} $(T,\lambda)$ of $E(G)$ is a branch decomposition $(T,\lambda^*)$ on $E(G^*)$ where for every $v\in L(T)$, $\lambda^*(v)$ is the dual edge of the edge $\lambda(v)$.

\subsection{Matroid branchwidth}
\label{subsec:mbw}
In this subsection, we define the \emph{matroid branchwidth} of a graph. Our definition is from~\cite{MazoitT07bran}.
See~\cite{KoutsonasTY14oute,Kashyap08matr} for the corresponding notion of \emph{matroid pathwidth}.

\paragraph{Cycle bases of graphs.}
A \emph{cycle basis} of a graph $G$ is a minimal set $C$ of cycles such that every Eulerian subgraph of $G$ (i.e., a subgraph containing a cycle which uses each edge of the subgraph exactly once) can be expressed as the symmetric difference of cycles in $C$. The number of cycles in every cycle basis of $G$ is constant and equal to $|E(G)|-|V(G)|+\mathsf{cc}(G)$, wherein $\mathsf{cc}(G)$ is the number of connected components of $G$. This number is called the \emph{circuit rank} of $G$. 

Given a hypergraph $G$ and a set $E\subseteq E(G)$,
\emph{a spanning forest of $E$} is a maximal acyclic subset of $E$.
We call \emph{size} of a spanning forest the number of edges it contains.
Given a graph $G$, we define the function $r:2^E \to \mathbb{N}$ such that for every $E\subseteq E(G)$, $r(E)$ is the size of a spanning forest of $E$.

\paragraph{Matroid branchwidth.}
Let $G$ be a hypergraph. We define the function $\mu_G: 2^{E(G)} \to \mathbb{N}$, 
where for every $F\subseteq E(G)$, $\mu_G(F) := r(F) + r(E(G)\setminus F) - r(E(G)) +1$.
The \emph{matroid branchwidth} of $G$, denoted by $\mu\text{-}\mathbf{bw}(G)$, is the branchwidth of $\mu_G$ on $E(G)$.
As observed in~\cite{MazoitT07bran}, while branchwidth and matroid branchwidth are two different parameters and, in general, $\mu\text{-}\mathbf{bw}(G)\leq \mathbf{bw}(G)$,~\autoref{prop:conn_bw} implies that in the case of connected bridgeless hypergraphs, branchwidth and matroid branchwidth  coincide.

\begin{proposition}\label{prop:bw-mbw}
For every connected bridgeless hypergraph $H$, $\mathbf{bw}(H) = \mu\text{-}\mathbf{bw}(H)$.
\end{proposition}

Note that, in matroid language, $r(E(G))$ is the rank of the graphic matroid of $G$ (i.e., the matroid whose ground set is $E(G)$ and independent sets are the edge sets of the forests of $G$).
Therefore, the matroid branchwidth of $G$ coincides with the branchwidth of the cycle matroid of $G$
(i.e., the matroid that has $E(G)$ as ground set, while its independent sets are the sub-forests of $G$).
This provides an immediate proof for the planar case of~\autoref{thm:main}, since the branchwidth of a planar matroid is equal to that of its dual. However, this cannot be extended to a general proof of our hypothesis in higher Euler genus graphs, as the only matroids that are at the same time graphic and co-graphic are planar matroids.

\subsection{Matroid branchwidth and graph operations}
\label{subsec:mbw-and-operations}

In this subsection we show how the matroid branchwidth of a graph changes when removing or contracting certain edges under different connectivity assumptions.
First, we give some definition and notation for edge removals and contractions on graphs.

\paragraph{Operations on graphs.}
Given a graph $G=(V,E)$ and a set $F\subseteq E$,
we use $G\setminus F$ to denote the graph $(V,E\setminus F)$ and, given an $e\in E$, we use $G\setminus e$ to denote the graph $(V,E\setminus \{e\})$.
We also denote by $G[F]$ the graph $(V\cap F^2, F)$.
The \emph{edge contraction} of an edge $e=\{u,v\}$ of a graph $G$ is the operation that deletes $e$, adds a new vertex $x_{uv}$ to $V$ and connects $x_{uv}$ to all the neighbours of $u$ and $v$ (notice that, as we consider multigraphs, this operation may create multiple edges, and that a contraction of an edge makes a triangle into a \emph{dipole}, and a dipole into a loop). We denote the graph obtained by $G/e$.
Given a set $F\subseteq E(G)$ and an edge $e=\{u,v\}\in F$,
we denote by $F^{/e}$ the set obtained from $F$ after the contraction of $e$, i.e., the set
$(F \setminus (\bigcup_{z\in N_G (u)\cup N_G (v)} \{\{u,z\},\{v,z\}\}))
\cup \bigcup_{z\in N_G (u)\cup N_G (v)}\{x_{uv},z\}$.

\paragraph{Contracting an edge of a connected graph.}
We observe that, given a connected graph $G$, the graph $G/e$ is still connected and has one fewer vertex than $G$.
Also, for every connected graph $H$ it holds that $r(E(H))=|V(H)|-1$.
This implies the following.
\begin{observation}\label{obs:conn}
If $G$ is a connected graph, then $r(E(G/e)) = r(E(G)) - 1$.
\end{observation}

We also prove that given a graph $G$ and a set $F\subseteq E(G)$ such that
both $G$ and $G[F]$ are connected, the contraction of an edge in $F$ can reduce $\mu_{G}(F)$ by one.

\begin{lemma}\label{lem:remove-loops}
Let $G$ be a connected graph embedded on a surface and let $F\subseteq E(G)$ such that $G[F]$ is connected.
For every $e \in F$, it holds that $\mu_{G}(F) -1 \leq \mu_{G/e}(F^{/e}) \leq \mu_{G}(F)$.
\end{lemma}

\begin{proof}
We show how $\mu_G (F)$ changes after the contraction of $e$.
Observe that since $G$ and $G[F]$ are connected, then, due to~\autoref{obs:conn}, $r(E(G/e)) = r(E(G)) - 1$ and $ r(F^{/e})= r(F)-1$.
To see how $r(E(G)\setminus F)$ changes, observe the following:
if $e\nsubseteq \partial(F)$, then $E(G)\setminus F = E(G/e)\setminus F^{/e}$ and therefore $r(E(G/ e)\setminus F^{/e}) = r(E(G)\setminus F)$.
If $e\subseteq \partial(F)$, note that the endpoints of $e$ may lie in the same component of $G\setminus F$
and therefore the size of a spanning forest of the graph obtained by $G\setminus F$ can be reduced by one after contracting $e$, implying $r(E(G)\setminus F)  -1 \leq r(E(G/ e)\setminus F^{/e})\leq r(E(G)\setminus F)$.
\end{proof}

\paragraph{Contracting non-loop edges of a connected bridgeless graph.}
In the next lemma we show that, given a graph $G$ that is connected and bridgeless (as the graphs in~\autoref{thm:main}) and a set $F\subseteq E(G)$, where $F$ and $E(G)\setminus F$ induce connected graphs, the contraction of a non-loop edge $e\in F$ reduces $\mu_{G}(F)$ by one if and only if both endpoints of $e$ belong to the border $\partial(F)$ of $F$. Otherwise, $\mu_{G}(F)$ remains the same.

\begin{lemma}\label{lem:contract}
Let $G$ be a connected graph and let $F\subseteq E(G)$, where both $G[F]$ and $G\setminus F$ are connected.
For every non-loop $e\in F$, it holds that
$$\mu_{G/e}(F^{/e})=\begin{cases}
 \mu_{G}(F) -1, & \text{if both endpoints of $e$ belong to $\partial(F)$ and}\\
\mu_{G}(F) , & \text{otherwise.}
\end{cases}
$$
\end{lemma}

\begin{proof}
Let $e$ be a non-loop $e\in F$.
We show how $\mu_G (F)$ changes after the contraction of $e$.
First, observe that since $G$ and $G[F]$ are connected,~\autoref{obs:conn} implies that $r(E(G/e)) = r(E(G)) - 1$ and $ r(F^{/e})= r(F)-1$. Finally, if both endpoints of $e$ are contained in $\partial(F)$, we have that $r(E(G/ e)\setminus F^{/e})  = r(E(G)\setminus F) -1$. To see this, observe that, since both endpoints of $e$ are contained in $\partial(F)$ (and therefore both endpoints are incident to edges in $E(G)\setminus F$) and $G\setminus F$ is connected, the graph obtained from $G\setminus F$ after the contraction of $e$ is connected and has one fewer vertex than $G\setminus F$.
If at least one of the endpoints of $e$ is not contained in $\partial(F)$, then the fact that $e\in F$ implies that the contraction of $e$ does not change the graph $G\setminus F$, and thus $r(E(G/ e)\setminus F^{/e}) = r(E(G)\setminus F)$.
\end{proof}

\paragraph{Removing non-bridge edges.}
We conclude this subsection by showing that, given a graph $G$ and a set $F\subseteq E(G)$, the removal of an edge $e\in F$ that is not a bridge of $G$ reduces $\mu_{G}(F)$ by one if and only if $e$ is a bridge of $G[F]$. Otherwise, $\mu_{G}(F)$ remains the same.

\begin{lemma}\label{lem:remove}
Let $G$ be a graph and let $F\subseteq E(G)$.
For every $e\in F$ that is not a bridge of $G$, it holds that
$$\mu_{G\setminus e}(F \setminus e)
=\begin{cases}
\mu_{G}(F)-1, & \text{if $e$ is a bridge of $G[F]$ and}\\
\mu_{G}(F), & \text{otherwise.}
\end{cases}
$$
\end{lemma}

\begin{proof}
Let $e$ be an edge in $F$ that is not a bridge of $G$.
We show how $\mu_G (F)$ changes after the removal of $e$.
In fact, we show that
\begin{enumerate}
\item $r(E(G\setminus e)) = r(E(G))$,
\item $ r(E(G\setminus e)\setminus (F\setminus e)) = r(E(G)\setminus F)$, and
\item $r(F\setminus e) = r(F) -1$ if $e$ is a bridge of $G[F]$, while $r(F\setminus e) = r(F)$, otherwise.
\end{enumerate}
Item (2) is directly implied by the fact that $e\in F$. Indeed, $r(E(G\setminus e)\setminus (F\setminus e)) =  r((E(G)\setminus \{e\})\setminus (F\setminus \{e\}) = r(E(G)\setminus F)$.
For items (1) and (3), recall that, by definition, for every set $B\subseteq E(G)$,
$r(B)$ is equal to the number of edges of a spanning forest of $G[B]$.
Let $B\subseteq E(G)$.
Observe that if $e$ is a bridge of $G[B]$ then every spanning forest $T$ of $G[B]$ contains $e$ and therefore every spanning forest $T'$ of $(G\setminus e) [B\setminus e]$ has one edge less than $T$, implying that $r(B\setminus e)  = r(B) -1$.
If $e$ is not a bridge of $G[B]$, then every spanning forest of  $(G\setminus e) [B\setminus e]$ is also a spanning forest of $G[B]$, implying that $r(B\setminus e) = r(B) $.
Item (1) follows by setting $B=E(G)$ and using the assumption that $e$ is not a bridge of $G$,
while item (3) follows by setting $B=F$.
\end{proof}

\subsection{On the duality of bridges and borders}
\label{subsec:dual-bridge-border}

We now aim to study when for a given set $F\subseteq E(G)$,
whether $e^*$ being a bridge of $G^*[F^*]$ implies that both endpoints of $e$ belong to $\partial(F)$, and vice-versa.
We will use the following result from~\cite{MoharT01grap}.

\begin{proposition}[\cite{MoharT01grap}] 
\label{prop:cycles}
Let $\Sigma$ be a surface of Euler genus $g$ and let $\{C_1,...,C_k\}$ be a set of distinct closed arcs of $\Sigma$. If $\Sigma\setminus\bigcup_{i\in [k]} C_i$ is connected, then $k\leq g$.
\end{proposition}

We show that if an $e\in F$ is non-loop and  $e^*$ is a bridge of $G^*[F^*]$, both endpoints of $e$ belong to $\partial(F)$.

\begin{lemma}\label{lem:bridge-sep-genus}
Let $G$ be a graph embedded on a surface.
Let $F\subseteq E(G)$ and let $e$ be a non-loop edge in $F$.
If $e^*$ is a bridge of $G^*[F^*]$, then both endpoints of $e$ belong to $\partial(F)$.
\end{lemma}

\begin{proof}
Suppose towards a contradiction that $e^*$ is a bridge of $G^*[F^*]$ and that there is an endpoint $v$ of $e$ such that $v\notin \partial(F)$.
Let $E$ be the set of edges of $G$ that have $v$ as an endpoint.
The fact that $v\notin \partial(F)$ implies that $E\subseteq F$.
Now notice that since $e$ is not a loop, $|E|\geq 2$.
Moreover, since $e^*$ is a bridge of $G^*[F^*]$, $e^*$ has two distinct endpoints (i.e., is not a loop), which we denote by $x$ and $y$.
Then, consider the graph $G^*[E^*\setminus \{e^*\}]$ and observe that it is connected.
Therefore, the graph $G^*[E^*]$, which is a subgraph of $G^*[F^*]$, has a cycle that contains $e^*$ in its edge set, a contradiction.
\end{proof}

Suppose that $e^*$ is not a bridge of $G^*[F^*]$,
and thus is contained in a cycle of $G^*[F^*]$. That does not necessarily mean that $e$ is in a cut set of $G$, hence it might well be that $e\subseteq \partial(F)$, without that meaning that $E(G)\setminus F$ has somehow become disconnected. However, we can prove the following.

\begin{lemma}\label{lem:sep-bridge-genus}
Let $G$ be a graph embedded on an surface of Euler genus $g$.
Let $F\subseteq E(G)$ such that $G\setminus F$ is connected.
If $G^*[F^*]$ has a cycle basis $\mathcal{C}$ where $|\mathcal{C}|>g$,
there is an
$e \in F$ such that $e^*$ belongs to a cycle of $\mathcal{C}$ and $e$ is not a subset of $\partial(F)$.
\end{lemma}

\begin{proof}
Let $\Sigma$ be the surface of Euler genus $g$ on which $G$ is embedded.
Let $\mathcal{C} = \{C_1,\ldots, C_\ell\},$ where $\ell > g$, be a cycle basis in $G^*[F^*]$.
According to \autoref{prop:cycles}, $\Sigma\setminus \bigcup_{i\in [\ell]} C_i$ is disconnected. Then there is an edge $e^*$ belonging to a cycle $C_i\in \mathcal{C}$ that bounds two distinct faces (that cycle might be $e^*$ itself, if $e^*$ is a loop) and let $f$ and $f'$ be the two distinct faces that are incident to the edge $e^*$ in $G^*$.
Keep in mind that there is no arc in $\Sigma\setminus \bigcup_{i\in [|\mathcal{C}|]} C_i$ starting from a point in $f$ and finishing to a point in $f'$.
Note that $e^*\in F^*$ and therefore $e\in F$. Let $u$ and $v$ be the endpoints of $e$, such that $v^* = f$ and $u^* = f'$.
We will show that $\{v,u\}$ is not a subset of $\partial(F)$.
Suppose towards a contradiction that both $u$ and $v$ are incident to edges of $E(G)\setminus F$.
Since $v,u\in V(G\setminus F)$ and $G\setminus F$ is connected, there should be a path
in $G\setminus F$ connecting $u$ and $v$.
This implies that there is an arc in $\Sigma$ starting from a point inside $u^* = f$ and finishing to a point inside $v^* = f'$ that does not cross any vertex of $G^*$ and any edge of $F^*$.
This contradicts the fact that $\Sigma\setminus \bigcup_{i\in [\ell]} C_i$ is disconnected and $f$ and $f'$ are in two distinct connected components of $\Sigma\setminus \bigcup_{i\in [\ell]} C_i$.
\end{proof}

Before concluding this subsection, we show that if $G$ is an embedded graph without loops, then its dual $G^*$ has no bridges.

\begin{lemma}
\label{lem:no-bridge-dual}
Let $G$ be a graph embedded on an surface.
If $G$ has no loops, then $G^*$ has no bridges.
\end{lemma}

\begin{proof}
Suppose, towards a contradiction, that $G$ has no loops and $G^*$ has a bridge $e^*$.
Let $f_1^*,f_2^*$ be the endpoints of $e^*$ and note that since $e^*$ is a bridge (and therefore not a loop) $f_1^*$ and $f_2^*$ are distinct and belong to different connected components of the graph $G^*\setminus e^*$.
Therefore, every path in $G^*$ connecting $f_1^*$ and $f_2^*$ should contain $e^*$.
This, in turn, implies that every arc of the surface whose endpoints are in the faces $f_1$ and $f_2$ of $G$ and does not cross any vertices of $G$, should cross $e$.
Thus, $e$ together with its endpoints form a closed arc, implying that $e$ is a loop, a contradiction.
\end{proof}

\subsection{\texorpdfstring{Proof of Theorem~\ref{thm:main}}{Proof of Theorem 1}}
\label{subsec:proof-graphs}

Let $G$ be a connected loopless bridgeless graph embedded on a surface of Euler genus $g$.
Observe that $G^*$ is also connected and since $G$ is loopless, $G^*$ is also bridgeless.
By~\autoref{prop:bw-mbw}, $\mathbf{bw}(G) = \mu\text{-}\mathbf{bw}(G)$ and $\mathbf{bw}(G^*) = \mu\text{-}\mathbf{bw}(G^*)$.
We consider a connected branch decomposition of $G$ of width $\mu\text{-}\mathbf{bw}(G)$.
This can be done using~\autoref{prop:conn_bw}.

Now, let $F\subseteq E(G)$ such that $G[F]$ and $G\setminus F$ are connected.
We inductively define a finite sequence of graphs $G_0,G_1,\ldots$ obtained as follows.
First, we set $G_0 := G$ and $F_0:=F$.
For every integer $i\geq 1$, if there is an edge $e\in F_{i-1}$ such that $e$ is not a loop in $G_{i-1}$ and its dual edge $e^*$ in $E(G_{i-1}^*)$ is a bridge of $G_{i-1}^*[F_{i-1}^*]$, we define $G_i:=G_{i-1}/e$ and $F_i:=F_{i-1}^{/e}$.
Observe that $G_i$ is bridgeless, connected, and embedded on $M$, and both $G_i[F_i]$ and $G_i\setminus F_i$
are connected.
Note also that since $e$ is not a loop in $G_{i-1}$, $e^*$ is not a bridge in $G_{i-1}^*$.
Moreover, the fact that $e^*$ is a bridge of $G_{i-1}^*[F_{i-1}^*]$ and~\autoref{lem:bridge-sep-genus} imply that
$e\subseteq \partial(F_{i-1})$.
Therefore, by~\autoref{lem:contract} and~\autoref{lem:remove} it holds that 
\begin{equation}\label{eq1}
|\mu_{G_{i-1}}(F_{i-1})-\mu_{G_{i-1}^*}(F_{i-1}^*)|=|\mu_{G_{i}}(F_{i})-\mu_{G_{i}^*}(F_{i}^*)|.
\end{equation}

Let $k\in\mathbb{N}$ such that for every edge $e\in F_k$ that is not a loop in $G_k$,
its dual edge $e^*$ is contained in some cycle of $G_k^*[F_k^*]$. By~\eqref{eq1}, we have that
$$|\mu_{G}(F)-\mu_{G^*}(F^*)|=|\mu_{G_k}(F_k)-\mu_{G_k^*}(F_k^*)|.$$

Now observe that if $G_k^*[F_k^*]$ has a cycle basis $\mathcal{C}$, where $|\mathcal{C}|>g$,
then by~\autoref{lem:sep-bridge-genus}, there exists an edge $e\in F_k$ such that 
$e^*$ belongs to a cycle of $\mathcal{C}$ and $e$ is not a subset of $\partial(F_k)$.
Therefore, by~\autoref{lem:contract} and~\autoref{lem:remove},
$\mu_{G_k}(F_k)=\mu_{G_k/e}(F_k^{/e})$ and $\mu_{G_k^*}(F_k^*)=\mu_{G_k^*\setminus e^*}(F_k^*\setminus e^*)$.
Also observe that the removal of $e^*$ from $G^*[F^*]$ reduces the circuit rank of $G^*[F^*]$ by one.
Applying iteratively the above argument, we can obtain a graph $\tilde{G}$ and a set $\tilde{F}$
such that $\tilde{G}$ is bridgeless, connected, and embedded on $M$, $\tilde{G}[\tilde{F}]$ and $\tilde{G}\setminus\tilde{F}$ are connected, 
\begin{equation}\label{eq:reduction}
|\mu_{G}(F)-\mu_{G^*}(F^*)|=|\mu_{\tilde{G}}(\tilde{F})-\mu_{\tilde{G}^*}(\tilde{F}^*)|,
\end{equation}
and the circuit rank of $\tilde{G}^*[\tilde{F}^*]$ is at most $g$.

Now, notice that since~\autoref{lem:sep-bridge-genus} cannot be applied for the graph $\tilde{G}$ and the set $\tilde{F}$, we not guaranteed the existence of a (non-loop) edge whose dual is not a bridge of $\tilde{G}^*[\tilde{F}^*]$ and $e\nsubseteq \partial(F)$.
Thus, by contracting a non-loop edge $e\in \tilde{F}$ of $\tilde{G}$ whose dual is not a bridge of $\tilde{G}^*[\tilde{F}^*]$,
due to~\autoref{lem:contract} and~\autoref{lem:remove}, we have that
$$\mu_{\tilde{G}^*\setminus e^*}(\tilde{F}^*\setminus e^*) = \mu_{\tilde{G}^*}(\tilde{F}^*) \mbox{~and~}
\mu_{\tilde{G}/e}(\tilde{F}^{/e}) = \begin{cases}
\mu_{\tilde{G}}(\tilde{F}) - 1, & \text{ if $e\subseteq \partial(F)$ and}\\
\mu_{\tilde{G}}(\tilde{F}) , & \text{ otherwise}.
\end{cases}$$
Note also that the removal of $e^*$ from $\tilde{G}^*[\tilde{F}^*]$ reduces the circuit rank of 
$\tilde{G}^*[\tilde{F}^*]$ by one.
Therefore, after iteratively contracting non-loop edges in the set $\tilde{F}$, we
obtain a connected (bridgeless) graph $\hat{G}$ embedded on $M$ and a set $\hat{F}\subseteq E(\hat{G})$
such that $\hat{G}[\hat{F}]$ is a single-vertex graph (whose edges are all loops), $\hat{G}\setminus \hat{F}$ is connected, and 
\begin{equation}\label{eq:daisy}
\mu_{\hat{G}^*}(\hat{F}^*) = \mu_{\tilde{G}^*}(\tilde{F}^*)
\mbox{~and~}
\mu_{\tilde{G}}(\tilde{F}) - r\leq \mu_{\hat{G}}(\hat{F})\leq \mu_{\tilde{G}}(\tilde{F}),
\end{equation}
for some $r\in\mathbb{N}$ such that $r\leq g$. Moreover, keep in mind that $r$ corresponds to the number of 
contracted non-loop edges in this step and therefore if $c$ is the circuit rank of $\hat{G}^*[\hat{F}^*]$, then $c+r\leq g$.

Observe now that since $\hat{G}^*[\hat{F}^*]$ has circuit rank $c\leq g$,
by~\autoref{lem:remove-loops} the contraction of the dual edge $e^*$ of a loop $e\in\hat{F}$ may reduce $\mu_{\hat{G}^*}(\hat{F}^*)$ by one.
Also, by~\autoref{lem:remove}, the removal of a loop in $\hat{F}$ (being a loop implies that it is neither a bridge of $\hat{G}$ nor of $\hat{G}[\hat{F}]$) does not reduce $\mu_{\hat{G}}(\hat{F})$.
Therefore, removing all remaining (loop) edges in $\hat{F}$, we 
obtain the graph $\breve{G}$ that is either empty, or connected and embedded in $M$,
such that
\begin{equation}\label{eq:empty}
\mu_{\hat{G}^*}(\hat{F}^*)-\ell\leq \mu_{\breve{G}^*}(\emptyset)\leq \mu_{\hat{G}^*}(\hat{F}^*)\text{~and~}\mu_{\breve{G}}(\emptyset) = \mu_{\hat{G}}(\hat{F})
\end{equation}
for some $\ell\in\mathbb{N}$ such that $\ell \leq c$.
Since $c+r\leq g$, we have that $\ell+r\leq g$, which also implies that $|\ell - r|\leq g$.
By combining~\eqref{eq:daisy} and~\eqref{eq:empty}, we obtain
$$
\mu_{\tilde{G}^*}(\tilde{F}^*) - \ell \leq \mu_{\breve{G}^*}(\emptyset)\leq\mu_{\tilde{G}^*}(\tilde{F}^*)\text{~and~}
\mu_{\tilde{G}}(\tilde{F}) - r\leq \mu_{\breve{G}}(\emptyset) \leq \mu_{\tilde{G}}(\tilde{F}),
$$
which in turn implies that
$$\big| |\mu_{\tilde{G}}(\tilde{F})-\mu_{\tilde{G}^*}(\tilde{F}^*)| - r +\ell \big| \leq  |\mu_{\breve{G}}(\emptyset)-\mu_{\breve{G}^*}(\emptyset)|.$$
This, together with the fact that $\mu_{\breve{G}}(\emptyset) = \mu_{\breve{G}^*}(\emptyset)= 1$,
implies that
$|\mu_{\tilde{G}}(\tilde{F})-\mu_{\tilde{G}^*}(\tilde{F}^*)|  \leq |\ell-r| \leq g.$
Moreover, using~\eqref{eq:reduction}, we get $$|\mu_{G}(F)-\mu_{G^*}(F^*)| \leq g.$$

Consider now a branch decomposition $(T,\lambda)$ of $G$ and its dual $(T,\lambda^*)$.
If $w_\mu (T,\lambda')= \mu_{G^*}(F^*)$ for some $F\subseteq E(G)$,
then $w_\mu (T,\lambda')= \mu_{G^*}(F^*) \leq \mu_{G}(F) +g \leq w_\mu(T,\lambda)+g$.
Hence, if $(T,\lambda)$ is an optimal branch decomposition for $G$, then $\mu\text{-}\mathbf{bw}(G)=w_\mu(T,\lambda)\geq w_\mu (T,\lambda')-g\geq \mu\text{-}\mathbf{bw}(G^*)-g$.
This, together with the fact that $\mu\text{-}\mathbf{bw}(G) = \mathbf{bw}(G)$ and $\mu\text{-}\mathbf{bw}(G^*) = \mathbf{bw}(G^*)$,
implies that $\mathbf{bw}(G^*)\leq \mathbf{bw}(G)+g$.

\section{From graphs to hypergraphs}

Before we proceed to the proof of~\autoref{main_result} we need to establish some more notation.

\subsection{Incidence graphs and their branch decompositions}
In this subsection we show how to obtain a branch decomposition of the incidence graph of a hypergraph from a branch decomposition of the hypergraph.
Then, we prove that the branchwidth of a hypergraph and its \emph{incidence graph} are the same (\autoref{lemma1}) and the same holds after taking the dual (\autoref{lemma2}).

\paragraph{Incidence graphs.}
For a given hypergraph $G$ we define its \emph{incidence graph}, which we denote by $I(G)$, the graph whose vertex set is $V(G)\cup E(G)$ and two vertices of $I(G)$ are adjacent if and only if one is associated with a vertex $v$ of $G$, the other is associated with a hyperedge $h$ of $G$, and $v$ is an endpoint of $h$.

Given a vertex $v_h$ of $I(G)$ associated with the hyperedge $h$ of $G$,
we use $E_I(h)$ to denote the set of edges of $I(G)$ that are incident to $v_h$.
Observe that for every edge $e$ of $I(G)$ there is a unique hyperedge $h\in E(G)$
such that $e\in E_I (h)$. Therefore, $\{E_I(h)\mid h\in E(G)\}$ is a partition of $E(I(G))$. 

\paragraph{Embedding incidence graphs.}
Let $G$ be a hypergraph embedded on a surface $\Sigma$. Notice that any given embedding of $G$ on $\Sigma$ naturally defines an embedding of $I(G)$ on $\Sigma$.
The vertices of $I(G)$ that are associated with vertices of $G$ are represented as the respective vertices of $G$, the vertices of $I(G)$ that are associated with hyperedges of $G$ are represented as the centres of the respective stars of $G$.
An edge $\{v_x,v_e\}$ of $I(G)$, where $v_x$ is the vertex of $I(G)$ that is associated with the vertex $x$ of $G$ and $v_e$ is the vertex of $I(G)$ that is associated with the hyperedge $e$ of $G$, is represented as the half edge that connects $x$ with the centre of $e$.
Given an embedding of $G$ on some surface $\Sigma$, we will always embed $I(G)$ on $\Sigma$ as above and we will make no distinction between the edges of $I(G)$ and the half-edges of $G$ to which the corresponding edges of $I(G)$ are associated.

\paragraph{Branch decompositions of incidence graphs.}
Let $(T,\lambda)$ be a branch decomposition of a hypergraph $G$. We will denote by $(\mathsf{Inc}(T),\mathsf{Inc}(\lambda))$ a branch decomposition of $I(G)$ that is obtained from $T$ as follows: first, for each edge $e$ of $G$, we construct a binary tree $T_e$ and we set $\lambda_e$ to be a bijection from $L(T_e)$ to $E_I(e)$.
We call such a tree a \emph{half-edge tree}.
Keep in mind that for every edge $e\in E(G)$, one can define many different trees $T_e$ and bijections $\lambda_e$ from $L(T_e)$ to $E_I(e)$.
For each $e\in E(G)$, consider a pair $(T_e,\lambda_e)$.
Then, pick an edge $\{u,v\}$ of $T_e$, remove it, add a new vertex $z_e$,
and make $z_e$ adjacent to $u$ and $v$. We denote this tree by $\tilde{T}_e$. 
We now define $\mathsf{Inc}(T)$ to be the tree obtained from $T$ after taking the disjoint union of $T$ and $\bigcup_{t\in L(T)} \tilde{T}_{\lambda(t)}$ and,
for every $t\in L(T)$, identifying $t$ with $z_{\lambda(t)}$.
Also, we set $\mathsf{Inc}(\lambda)$ be the disjoint union of the bijections $\lambda_e, e\in E(G)$, i.e., for every $t\in L(\mathsf{Inc}(T))$, if $t\in L(T_e)$ for some $e\in E(G)$, then $\mathsf{Inc}(\lambda)$ maps $t$ to $\lambda_{e}(t)$.
\medskip

To prove \autoref{main_result} we first show the following intermediate results.

\begin{lemma}
\label{lemma1}
Let $(T,\lambda)$ be a branch decomposition of a hypergraph $G$. Then $w_\delta(T,\lambda) = w_\delta(\mathsf{Inc}(T),\mathsf{Inc}(\lambda))$.
\end{lemma}

\begin{proof} 
Recall that by definition, the edges of $\mathsf{Inc}(T)$ are partitioned to the set $E(T)$ and the edge sets of the half-edge trees $E(T_h)$, for all $h\in E(G)$.
Let $e$ be an edge of $\mathsf{Inc}(T)$. 
There is an $F\subseteq \bigcup_{h\in E(G)}E_I(h)$ such that $\delta_{\mathsf{Inc}(\lambda)}(e)=|\partial(F)|$.
We distinguish cases, depending whether $e\in E(T)$ or $e\in E(T_h)$, for some $h\in E(G)$.

If $e\in E(T)$, then notice that if for some hyperedge $y\in E(G)$ we have $x\in E_I(y)$ and $x\in F$ then $E_I(y)\subseteq F$ and if $x\notin F$, then then $E_I(y)\cap F = \emptyset$.
Therefore, there is a set $D\subseteq E(G)$ such that $F = \bigcup_{h\in D} E_I(h)$.
Observe that the union of the closure of all arcs in $F$ and the union of the closure of all arcs in $D$ define the same subsets of $\Sigma$.
Hence, $\delta_{\mathsf{Inc}(\lambda)}(e)=|\partial(F)|=|\partial(D)|=\delta_{\lambda}(e)$.

If $e$ belongs to $T_h$ for some $h\in E(G)$,
then $F\subseteq E_I(h)$. Moreover, notice that $\partial(F)$ is equal to a proper subset of the set of endpoints of the edges in $E_I(h)$ that contains the common endpoint of all edges in $E_I(h)$ (intuitively, the center of a star together with some of its degree-one vertices).
Therefore, $\delta_{\mathsf{Inc}(\lambda)}(e) = |\partial(F)|<|\partial(E_I(h))|=\delta_{\mathsf{Inc}(\lambda)}(e')$ for $e'$ being the edge of $\mathsf{Inc}(T)$ in $E(T)$ that is incident to the root of $T_h$.

Thus, $w_\delta(\mathsf{Inc}(T),\mathsf{Inc}(\lambda))$ is obtained for a partition of $E(I(G))$ defined by an edge $e\in E(T)$, and in this case $\delta_{\mathsf{Inc}(\lambda)}(e) = \delta_{\lambda}(e)$.
\end{proof}

\begin{lemma}
\label{lemma2}
Let $(T,\lambda)$ be a branch decomposition of a hypergraph $G$ and let $(T,\lambda^*)$ be its dual branch decomposition. Then $w_\delta(\mathsf{Inc}(T),\mathsf{Inc}(\lambda)^*)= w_\delta(\mathsf{Inc}(T),\mathsf{Inc}(\lambda^*))$.
\end{lemma}

\begin{proof}
We will show that $w_\delta(\mathsf{Inc}(T),\mathsf{Inc}(\lambda)^*)=\delta_{\mathsf{Inc}(\lambda)^*}(e)$ for some $e\in E(T)$. This will imply that 
$w_\delta(\mathsf{Inc}(T),\mathsf{Inc}(\lambda)^*)= w_\delta(\mathsf{Inc}(T),\mathsf{Inc}(\lambda^*))$.

We observe that, for a star $h$ of $G$, $E_I(h)^*$ is a cycle of $I(G)^*$ with as many edges as $E_I(h)$. Let $e$ be an edge of $\mathsf{Inc}(T)$ that belongs to a subtree $T_h$,
for some $h\in E(G)$. Then $\delta_{\mathsf{Inc}(\lambda)^*}(e)=|\partial(F)|$, wherein $F\subseteq E(h)^*$.
Thus,  $|\partial(F)|\leq \delta_{\mathsf{Inc}(\lambda)^*}(e')$, wherein $e'$ is the edge of $\mathsf{Inc}(T)$ that is in $E(T)$ and is incident to the root of $T_h$.

We will show that for every $e\in T$, $\delta_{\mathsf{Inc}(\lambda^*)}(e) = \delta_{\mathsf{Inc}(\lambda)^*}(e)$.
First let $F\subseteq E(I(G^*))$ such that $\delta_{\mathsf{Inc}(\lambda^*)}(e) = |\partial(F^*)|$.
Let also $F'\subseteq E(I(G)^*)$ such that $\delta_{\mathsf{Inc}(\lambda)^*}(e) = |\partial(F')|$.
We will prove that $\partial(F) = \partial(F')$.
The fact that $e\in T$ implies that there is some $D\subseteq E(G^*)$ and some $D'\subseteq E(G)$
such that $$F= \bigcup_{h^*\in D} \{z\mid z\in E_I(h^*)\}\text{~and~}F' = \bigcup_{h\in D'} \{z^*\mid z\in E_I (h)\}.$$
Note that $D^* = D'$ and therefore, for every $h^*\in D$, since the endpoints of the
star (in $I(G^*)$) corresponding to $h^*$ are the same as the endpoints of the duals of the edges of the star (in $I(G)$) corresponding to $h$, i.e., $\bigcup_{z\in E_I(h^*)} z = \bigcup_{z\in E_I(h)} z^*$.
Therefore, $\partial(F)\cap \bigcup_{z\in E_I(h^*)} z = \partial(F')\cap \bigcup_{z\in E_I(h)} z^*$, which in turn implies that $\partial(F) = \partial(F')$. Hence, $\delta_{\mathsf{Inc}(\lambda^*)}(e) = \delta_{\mathsf{Inc}(\lambda)^*}(e)$.
\end{proof}

\subsection{\texorpdfstring{Proof of~\autoref{main_result}}{Proof of Theorem 2}}

Let an optimal branch decomposition of a hypergraph $G$, say $(T,\lambda)$, such that $\mathbf{bw}(G)=w_\delta(T,\lambda)$.
Therefore, using~\autoref{lemma1}, we derive that $\mathbf{bw}(G)=w_\delta(\mathsf{Inc}(T),\mathsf{Inc}(\lambda))$.
Observe that by applying~\autoref{thm:main} to $I(G)$,
we have that $w_\delta(\mathsf{Inc}(T),\mathsf{Inc}(\lambda))\geq w_\delta(\mathsf{Inc}(T),\mathsf{Inc}(\lambda)^*)-\mathbf{eg}(G)$.
Therefore, it holds that $\mathbf{bw}(G)\geq w_\delta(\mathsf{Inc}(T),\mathsf{Inc}(\lambda)^*)-\mathbf{eg}(G)$.
Following~\autoref{lemma2}, we get that 
$$\mathbf{bw}(G)\geq w_\delta(\mathsf{Inc}(T),\mathsf{Inc}(\lambda^*))-\mathbf{eg}(G).$$
This together with the fact that $w_\delta(\mathsf{Inc}(T),\mathsf{Inc}(\lambda^*)) \geq \mathbf{bw}(G^*)$ implies that $\mathbf{bw}(G^*)\leq \mathbf{bw}(G)+\mathbf{eg}(G)$.\medskip

\paragraph{Acknowledgements.}
We are all thankful to Dimitrios M. Thilikos for communicating us the problem and inviting us to write its proof. The valuable guidance, inspiration, and knowledge that he generously offered, helped us in many different stages of this project. 

{\small 
\bibliographystyle{plainurl}

}

\end{document}